\documentclass[12pt]{amsart}
\usepackage{snowshovelingalg} 

\SetKwInOut{Parameter}{Parameters}

\newcommand\numberthis{\addtocounter{equation}{1}\tag{\theequation}}

\let\oldnl\nl
\newcommand{\nonl}{\renewcommand{\nl}{\let\nl\oldnl}}

\begin{document}

\title[Computational Optimal Transport with General Storage Fees]{ Computational Semi-Discrete Optimal Transport with General Storage Fees}

\author{Mohit Bansil}

\begin{abstract}
	
We propose and analyze a modified damped Newton algorithm to solve the semi-discrete optimal transport with storage fees. We prove global linear convergence for a wide range of storage fee functions, the main assumption being that each warehouse's storage costs are independent. We show that if $F$ is an arbitrary storage fee function that satisfies this independence condition then $F$ can be perturbed into a new storage fee function so that our algorithm converges. We also show that the optimizers are stable under these perturbations. Furthermore, our results come with quantitative rates. 
	
\end{abstract}
 
\maketitle

\tableofcontents

\section{Introduction}\label{section: intro}

\subsection{Semi-discrete optimal transport with storage fees} 
In this paper we present an algorithm to compute numerical solutions to the semi-discrete optimal transport problem with storage fees. This problem can be described as follows. Let $X\subset \R^n$, $n\geq 2$ be compact and $Y:=\{y_i\}_{i=1}^N \subset \R^n$ a fixed collection of finite points, along with a \emph{cost function} $c: X\times Y\to \R$ and a \emph{storage fee function} $F: \R^N\to\R$. We also fix a Borel probability measure $\mu$ with $\spt \mu\subset X$, and assume $\mu$ is absolutely continuous with respect to Lebesgue measure.

We want to find a pair $(T, \weightvect)$ with $\weightvect=(\weightvect^1,\ldots, \weightvect^N)\in \R^N$ and $T: X\to Y$ measurable satisfying
\begin{align*}
T_\#\mu = \sum_{i=1}^N \weightvect^i \delta_{y_i}
\end{align*}
such that
\begin{align}\label{eqn: monge ver}
\int_X c(x, T(x)) d\mu + F(\weightvect) = \min_{\tilde \weightvect\in \R^N,\ \tilde{T}_\#\mu = \sum_{i=1}^N \tilde\weightvect^i \delta_{y_i}} \int_X c(x, \tilde{T}(x)) d\mu + F(\tilde\weightvect).
\end{align}
We remark that by taking $F$ to be the indicator function of a point, we cover the classical semi-discrete optimal transport problem where we are given fixed probability measures $\mu, \nu$ (with $\mu$ absolutely continuous and $\nu$ discrete) and we want to find a measurable map $T$ so that $T_\#\mu = \nu$ and 
\begin{align}\label{eqn: monge classical}
\int_X c(x, T(x)) d\mu = \min_{\tilde{T}_\#\mu = \nu} \int_X c(x, \tilde{T}(x)) d\mu.
\end{align}
This problem was first studied in \cite{Crippa2009} in the case where $F(\wv) = \sum_i h_i(\wv^i)$ for some $h_i$, a condition we refer to as ``the storage fee function splitting''. This condition can be though of as the storage costs between different warehouses being independent. In that setting the authors showed existence and uniqueness under some regularity and convexity conditions and gave a characterization of the optimizer. The problem with non-splitting storage fees in analyzed in \cite{BansilKitagawa19a} where the authors found a dual problem with strong duality:
\begin{align*}
\min_{\tilde \weightvect\in \R^N,\ \tilde{T}_\#\mu = \sum_{i=1}^N \tilde\weightvect^i \delta_{y_i}} \int_X c(x, \tilde{T}(x)) d\mu + F(\tilde\weightvect)
=\sup_{\psi \in \R^N} \int_X \min_{i} c(x,y_i) + \psi^i \dmu - F^*(\psi).
\end{align*}
Furthermore, it is shown that given a dual maximizer, $\psi$, the minimizing transport map $T$ can be constructed by sending each point of $X$ to the warehouse of its corresponding Laguerre cell of the Laguerre partition generated by $\psi$. The minimizing $\wv$ is then seen to be given by $\wv^i = T_\#\mu(\{y_i\})$. 

\subsection{Informal Overview of Results}

In this paper we propose a modified damped Newton method to solve the dual problem and hence construct approximate solutions to the primal problem. We will show global linear convergence and local superlinear convergence along with a quantitative rate. 

We will require two assumptions on the storage fee function $F$. The first major assumption is a technical condition that will be satisfied whenever the storage fee function splits, i.e. $F(\wv) = \sum_i f_i(\wv^i)$ for some $f_i: \R \to \R$. The second major assumption is that $F^*$ needs satisfy a regularity, strong convexity, and non-degeneracy assumption. 

In order to show that this second assumption is not too constraining, we give a method to perturb any convex $F$ that splits into one that satisfies the assumption. We are then concerned with how this perturbation will effect the associated optimal transport map. To address this, we prove a stability result on the optimizers under perturbations of $F$ which may be interesting on its own. With this our algorithm can find approximate solutions for the semi-discrete optimal transport with storage fees for any convex $F$ that splits, in particular for all of the storage fee functions analyzed in \cite{Crippa2009}. 

The major difficultly obtaining convergence is that the functional in the dual problem is only well-conditioned when all of the Laguerre cells have positive mass. Unfortunately, if the initial guess is not very close to the actual solution, a Newton step might end up collapsing a Laguerre cell. In other works on computational semi-discrete optimal transport (such as \cite{Merigot2018, KitagawaMerigotThibert19}) this was addressed by damping the Newton steps. In the classical case, it turns out that if the cells of an approximate solution are already too small then any sufficiently small step will not make them smaller. Unfortunately in the setting with storage fees this is not true and so if we only use damping to keep the cells from collapsing we will not get global linear convergence. 

To remedy this we introduce a sub-routine that we call parameter shuffling. Given any approximate solution in which some Laguerre cells are too small, this routine finds another approximate solution in which the sizes of all of the Laguerre cells are bounded from below and for which the error is not larger then the initial error. Our algorithm functions by alternating between Newton steps which reduce the error but might make some Laguerre cells too small and parameter shuffling steps which fix the Laguerre cells and do not increase the error. 

We remark that our parameter shuffling routine might be of independent interest in that it provides a means to find good initial guesses for any semi-discrete optimal transport algorithm that requires non-empty Laguerre cells (for example the algorithms of \cite{KitagawaMerigotThibert19, Merigot2018, BansilKitagawa20b}). 

\subsection{Literature Review}

The semi-discrete optimal transport problem with storage fees was first posed in \cite{Crippa2009}. In this paper they consider the case where the storage function splits and under some regularity and convexity assumptions they prove that there is a unique optimizer and give a characterization of it. The problem is analyzed in greater generality in \cite{BansilKitagawa19a} where the authors also provide a dual problem with strong duality. 

The use of Newton-type algorithms in the semi-discrete setting seems to first appear in \cite{OlikerPrussner88}. Here the authors prove local convergence of a Newton algorithm to solve a semi-discrete Monge-Amp\`ere equation with Dirichlet boundary conditions. Global convergence is established in \cite{Mirebeau15} although without any quantitative rates. 

Concerning the classical semi-discrete optimal transport problem, experimentally fast algorithms are presented in \cite{Merigot11, Levy15}  however they do not come with a convergence guarantee. A damped Newton algorithm with a proven quantitative rate of convergence was developed in \cite{KitagawaMerigotThibert19}. This idea is extended in \cite{Merigot2018} to solve the optimal transport problem in the case where the source is supported on a union of simplices and the target is discrete. An overview of numerics for the semi-discrete optimal transport problem is given in \cite[Section 6.4.2]{Santambrogio15}. 

Concerning numerics for the semi-discrete optimal transport problem with storage fees to the best of the author's knowledge the only previous algorithm is that of \cite{BansilKitagawa20b}. This paper only treats a very specific case of storage functions and not only the proof of convergence but also the Newton algorithm itself heavily depends on the specific choice of storage fee function. In particular the algorithm does not actually directly maximize the dual problem by searching for the zero of its gradient (indeed the mapping for which a zero is found is not the gradient of any scalar function, see \cite[Remark 2.7]{BansilKitagawa20b}) and so our algorithm is fundamentally different from the one presented there. Concerning convergence, that algorithm does have global linear and local superlinear convergence of the same order that ours does.  

\section{Setup}\label{section: setup}

\subsection{Notations and Conventions}
In this subsection we collect some notations and conventions that will be used throughout the entire paper. We fix positive integers $N$ and $n$ and a collection $Y:=\{y_i\}_{i=1}^N\subset \R^n$. For any vector $V \in \R^k$, we will write its components as superscripts so $V^i$ is the $i$-th component of $V$. We reserve the notation $\onevect$ to refer to the vector in $\R^N$ whose components are all $1$. We use $\norm{V}_p$ for the $l^p$ Euclidean norm, i.e. $\norm{V}_p = \sum_{i=1}^k \abs{V^i}^p$. We will use $\norm{V}$ to refer to the standard Euclidean norm, $\norm{V}_2$. 
We shall use that notation
\begin{align*}
\weightvectset:=\{{\weightvect}\in \R^N\mid \sum_{i=1}^N\weightvect^i=1,\ \weightvect^i\geq 0\},
\end{align*}
for the set of admissible weight vectors.

For any subsets $A, B \subset \R^k$ we shall use $d_\H(A,B)$ to denote the Hausdorff distance between $A$ and $B$. We adopt the notation $\delta_A$ to denote the indicator function of $A$ i.e.
\begin{align*}
\delta_A(x) =
\begin{cases}
0, & \text{ if } x \in A \\
+\infty & \text {else}
\end{cases}.
\end{align*}
We shall assume that the cost function $c$ satisfies the following standard conditions:
\begin{align}
c(\cdot, y_i)&\in C^2(X), \forall i\in \{1, \ldots, N\},\label{Reg}\tag{Reg}\\
\nabla_xc(x, y_i)&\neq \nabla_xc(x, y_k),\ \forall x\in X,\ i\neq k.\label{Twist}\tag{Twist}
\end{align}
We also assume the following condition, originally studied by Loeper in \cite{Loeper09}.
\begin{defin}
	The cost function, $c$, is said to satisfy \emph{Loeper's condition} if for each $i\in \{1, \ldots, N\}$ there exists a convex set $Y_i\subset \R^n$ and a $C^2$ diffeomorphism $\cExp{i}{\cdot}: Y_i\to X$ such that 
	\begin{align*}
	\forall\ t\in\R,\ 1\leq k, i\leq N,\ \{p\in Y_i\mid -c(\cExp{i}{p}, y_k)+c(\cExp{i}{p}, y_i)\leq t\}\text{ is convex}.\label{QC}\tag{QC}
	\end{align*}
	See Remark \ref{rmk: brenier solutions} below for further discussion of these conditions.
	
	We also say that a set $X\subset \R^n$ is \emph{$c$-convex} with respect to $Y$ if $\invcExp{i}{X}$ is a convex set for every $i\in \{1, \ldots, N\}$. 
\end{defin}

\begin{defin}
	For any $\psi\in \R^N$ and $i\in \{1, \ldots, N\}$, we define the \emph{$i$th Laguerre cell} associated to $\psi$ as the set
	\begin{align*}
	\Lag_i(\psi):=\{x\in X\mid c(x, y_i)+\psi^i= \min_{i} c(x,y_i) + \psi^i \}.
	\end{align*}
	We also define the function $G: \R^n\to \weightvectset$ by
	\begin{align*}
	G(\psi):=(G^1(\psi), \ldots, G^N(\psi))=(\mu(\Lag_1(\psi)), \ldots, \mu(\Lag_N(\psi))),
	\end{align*}
	and denote for any $\epsilon\geq 0$,
	\begin{align*}
	\mathcal{K}^\epsilon:=\{\psi\in \R^N\mid G^i(\psi)> \epsilon,\ \forall i\in \{1, \ldots, N\}\}.
	\end{align*}
\end{defin}

\begin{rmk}\label{rmk: brenier solutions}
Of the above conditions on the cost, \eqref{Reg} and \eqref{Twist} are standard conditions in the existence theory for optimal transport. Furthermore, \eqref{QC} holds if $Y$ is a finite set sampled from from a continuous space, and $c$ is a $C^4$ cost function satisfying what is known as the \emph{Ma-Trudinger-Wang} condition (first introduced in a strong form in \cite{MaTrudingerWang05}, and in \cite{TrudingerWang09} in a weaker form).
	
	If $\mu$ is absolutely continuous with respect to Lebesgue measure, then the condition \eqref{Twist} implies that the Laguerre cells are pairwise $\mu$-almost disjoint. In this case the generalized Brenier's theorem \cite[Theorem 10.28]{Villani09}, tells us that for any $\psi\in \R^N$, the map $T_\psi: X\to Y$ defined by $T_\psi(x)=y_i$ whenever $x\in \Lag_i(\psi)$ is a minimizer in the optimal transport problem \eqref{eqn: monge classical}, where the source measure is $\mu$ and the target measure is defined by $\nu(\{y_i\})=G(\psi)^i$.
\end{rmk}

For the remainder of this paper we will also assume that $X$ is compact and $c$-convex with respect to $Y$. Furthermore we denote the density of the (absolutely continuous) measure $\mu$ by $\rho$ and we assume that $\rho$ is $\alpha$-H\"older continuous for some $\alpha \in (0, 1]$.

Next we set some definitions concerning the function $F$, representing the cost of warehouse storage.

\begin{defin}
A storage fee function, $F$, is a proper closed convex function from $\R^n \to \R \cup \{+\infty\}$ that is $+\infty$ outside of $\wvs$ (in other words $\dom F \subset \wvs$, where $\dom F$ means the effective domain of a convex function).
\end{defin}

 These are the sufficient conditions for strong duality given in \cite{BansilKitagawa19a}.  Throughout this note $F$ will always be a storage fee function unless otherwise noted. 

\begin{defin}
	A storage fee function $F$ is said to split if there are $f_i: \R \to \R$ so that 
	\begin{align*}
	F(\wv) = \sum_i f_i(\wv^i) + \delta_\wvs(\wv)
	\end{align*}
	where each $f_i$ is convex, closed, and proper. Furthermore we require that $\dom f_i \subset [0,1]$ (if this is not true we can replace $f_i$ with $f_i + \delta_{[0,1]}$). 
\end{defin}

We remark that the condition that $\dom F \subset \wvs$ (and so the addition of the $\delta_\wvs(\wv)$ term) does not add any additional assumption. Since in the optimal transport problem with storage fees we minimize over pairs $(T, \wv)$ satisfying ${T}_\#\mu = \sum_{i=1}^N \weightvect^i \delta_{y_i}$, we must have that $\wv \in \wvs$ since $\mu$ was a probability measure. In particular solving the optimal transport problem with storage fee function $F + \delta_\wvs$ will yield the exact same solution as solving it with storage fee function $F$. 

For our algorithm to converge we will need some kind of strong convexity of the objective functional. There are two ways that we can get this. Either we can assume that $F^*$ has some kind of strong convexity which corresponds to some kind of regularity of $F$. Alternatively, we can require that the support of $\mu$ satisfies some kind of quantitative connectedness assumption. It turns out that for our purposes it suffices to assume that $\mu$ satisfies a Poincar\'e-Wirtinger inequality. 

\begin{defin}
	A probability measure $\mu$ on $X$ satisfies a \emph{Poincar\'e-Wirtinger inequality}  if there is a constant $\Cpw>0$ such that for any $f\in C^1(X)$,
	\begin{align*}
	\norm{f-\int_Xfd\mu}_{L^1(\mu)}\leq \Cpw \norm{\nabla f}_{L^1(\mu)}.
	\end{align*}
	If this holds we will say that ``$\mu$ satisfies a \emph{PW inequality}''.
\end{defin}

We remark that if the support of $\mu$ is connected and $\rho$ is bounded away from $0$ on support of $\mu$, then it is classical that $\mu$ satisfies a PW inequality. 

Finally we call the objective functional to be maximized $\Phi$ defined as
\begin{align*}
\Phi(\psi) 
:= \int_X \min_{i} c(x,y_i) + \psi^i \dmu - F^*(\psi)
= \sum_{i} \int_{\Lag_i(\psi)} c(x, y_i) + \psi^i \dmu - F^*(\psi).
\end{align*}
We recall that $\nabla \Phi = G - \nabla F^*$, see for example \cite[Section 6.4.2]{Santambrogio15}. The conditions \eqref{Reg}, \eqref{Twist}, \eqref{QC} are sufficient to obtain the $C^{1,\alpha}$ regularity of $G$ and a $PW$ inequality is sufficient to obtain strong monotonicity of $G$ outside of the direction $\onevect$ as long as $\psi \in \mathcal{K}^\eps$ for some $\eps > 0$ (see \cite[Theorems 4.1, 5.1]{KitagawaMerigotThibert19}).

\subsection{Statement of Main Results}
We are now ready to state our algorithm and main results. We start by describing our parameter shuffling routine. 

\begin{algorithm}
	\DontPrintSemicolon
	
	\KwIn{A tolerance $\eps \in (0, \frac{1}{3N})$ and an initial $\psi_{in} \in \R^N$.}
	
	Set $\psi := \psi_{in}$\;
	
	\While{$\min_j G^j(\psi) \leq \eps$}
	{
		\For{$i \in \{1, \dots, N \}$}
		{
			\If{$G^i(\psi) \leq \eps$ \label{line: MF if}} 
			{			
				Find $r > 0$ so that $G^i(\psi - r e^i ) \in [2\eps, 3\eps]$ where $e^i$ is the $i$-th standard coordinate. \label{line: MJ find}\;
				Set $\psi = \psi - r e^i$ \label{line: MJ increment} \;
			}
		}
	}
	
	\Return{$\psi_{out} := \psi$}
	
	\caption{Parameter Shuffling Routine}
	\label{alg: magic juggle}
\end{algorithm}

We remark that since $G$ is a monotone function (it is the gradient of a concave function) we can always use a binary search to find the $r$ needed in line \ref{line: MJ find}. Next we describe our modified damped Newton algorithm.

\begin{algorithm}[H]
	
	\DontPrintSemicolon
	\LinesNumberedHidden 
	
	\KwIn{A tolerance $\zeta > 0$, an initial $\psi_0\in
		\R^N$, and an $\eps > 0$ such that $\nabla F^* \geq \eps$ coordinate-wise on $\mathcal{K}^0$. Set $\eps_0 = \frac{\eps}{4}$. } 
	
	\While{$\norm{\nabla  \Phi(\psi_k)}  \geq \zeta$}
	{
		\begin{description}
			
			\item [Step 1] Run the Parameter Shuffling Routine on $\psi_{k}$ with tolerance parameter $2\eps_0$. 
			
			\item[Step 2] Compute $\vec{d}_k = - [D^2  \Phi(\psi_k)]^{-1} (\nabla  \Phi(\psi_k))$
			\item [Step 3] Determine the minimum $\ell \in \N$ such that $\psi_{k+1, \ell} :=	\psi_k + 2^{-\ell} \vec{d}_k $ satisfies
			\begin{equation*}
			\left\{
			\begin{aligned}
			&\min_i G^i(\psi_{k+1, \ell}) \geq \eps_0 \\
			&\norm{\nabla  \Phi(\psi_{k+1, \ell})}_1 \leq (1-2^{-(\ell+1)}) \norm{\nabla  \Phi(\psi_k)}_1
			\end{aligned}
			\right.
			\end{equation*}
			\item [Step 4] Set $\psi_{k+1} = \psi_k + 2^{-\ell}  \vec{d}_k$ and $k\gets k+1$.

		\end{description}	
	}	
	\caption{Damped Newton's algorithm}
	
	\label{alg: damped newton}
\end{algorithm}

We remark that $D^2  \Phi(\psi_k) = DG - D^2F^*$. The matrix $DG(\psi)$ can be explicitly computed in terms of the Laguerre diagram associated to $\psi$ (see the discussion before Lemma 6.5 in \cite{Santambrogio15}) and under certain assumptions we will explicitly compute $D^2F^*$ in terms of $F$ (see the proof of Theorem \ref{thm: smoothness of F*}). 

Also note that Algorithm \ref{alg: damped newton} will always keep the sizes of the cells bigger than $\frac{\eps}{4}$. The condition that $\nabla F^* \geq \eps$ coordinate-wise on $\mathcal{K}^0$ insures that the cells in the true optimal solution are bigger than ${\eps}$ and so that algorithm can converge. We show in Theorem \ref{thm: regularize} that this assumption on $\nabla F^*$ is not too restrictive.  

Our first main theorem is that under certain conditions on $F$, Algorithm \ref{alg: damped newton} has has global linear convergence and local superlinear convergence. In particular under these assumptions we have that $\nabla F^* \geq \eps$ coordinate-wise. 

\begin{thm}\label{thm: Newton Convergence}
	
Suppose $F$ is a storage fee function that splits into $f_i$ so that the $f_i$ are essentially smooth and twice continuously differentiable on their domains, the $f_i''$ are locally Lipschitz on their domains, and each $f_i$ is strongly convex. Furthermore if $a_i, b_i$ are such that $ \conj {\dom f_i} = [a_i, b_i]$ we assume that there is some $\eps > 0$ so that $a_i > \eps$ and that $\sum_i a_i < 1 < \sum_i b_i$. 
Then Algorithm \ref{alg: damped newton} has global linear convergence and local superlinear convergence of order $\alpha$, the H\"older constant of $\rho$. 
\end{thm}

For our next theorem we show that the assumptions imposed in our main convergence theorem are not too restrictive in the sense that for any storage fee function that splits, there is an approximating storage fee function that satisfies the assumptions. 

\begin{thm}	\label{thm: regularize}
	Let $F$ be a storage fee function that splits into $f_i$. Then for every $\eta > 0$ there is a storage fee function $\ti F$ (explicitly constructed in the proof) so that $\ti F$ satisfies the assumptions of Theorem \ref{thm: Newton Convergence} and if $\wv, \ti \wv$ are the minimizers of problems associated to $F, \ti F$ respectively then $\norm{\wv - \ti \wv} \leq \eta$.
\end{thm}

We remark here that the results of \cite{BansilKitagawa20a} tell us that the optimal transport maps that solve the problems associated to $F$ and $\ti F$ are also close in the sense of $L^1(\mu)$ distance. 

\subsection{Outline of Paper}

In section \ref{sec: Parameter Shuffling} we show that Algorithm \ref{alg: magic juggle} terminates and does not increase error. In section \ref{sec: Convergence of Newton Algorithm} we prove global linear and local superlinear convergence of Algorithm \ref{alg: damped newton} under assumptions given in terms of $F^*$. In section \ref{sec: Relationship between F and F*} we translate the assumptions on $F^*$ back to assumptions on $F$ and use this to prove Theorem \ref{thm: Newton Convergence}, our main convergence theorem. In section \ref{sec: Stability} we obtain some results on the stability of the optimizing weight vector under perturbations in the storage fee function. In section \ref{sec: Regularizations} we show how to regularize any splitting storage fee function into one that satisfies the assumptions of Theorem \ref{thm: Newton Convergence}. This then gives a proof of Theorem \ref{thm: regularize}. We remark that in this section we also obtain an explicit formula for $D^2F^*$ that may be useful for implementing Algorithm \ref{alg: damped newton}. Finally in appendix \ref{sec: Appendix Bounds on psi's} we prove a quick lemma that bounds the difference between different coordinates of any $\psi \in \mathcal{K}^0$.

\section{Parameter Shuffling Analysis}\label{sec: Parameter Shuffling}

In this section we analyze Algorithm \ref{alg: magic juggle}. Our first proposition shows that it always terminates and gives a bound of how many iterations it can take. 

\begin{prop}
Algorithm \ref{alg: magic juggle} terminates in at most $(N-1) (\frac{4L\norm{c}_\infty}{\eps}+ 1)$ steps where $L$ is the Lipschitz constant of $G$. 
\end{prop}

\begin{proof}

First of all we claim it is not possible for every $\psi^i$ to be increased. More rigorously let $A \subset \{1, \dots, N \}$ be the collection of incidences for which the ``if'' statement in line \ref{line: MF if} evaluates as true in some iteration. We claim that $A \neq \{1, \dots, N \}$. To see this note that if $k \in A$ then $G(\psi)^k \leq 3\eps$ throughout the entire algorithm after line \ref{line: MF if} evaluates as true for the index $k$. Since $\sum_i G(\psi)^i = 1 > 3 N \eps$ it is not possible for $A = \{1, \dots, N \}$. Hence we can fix some $j \not\in A$ and some $k \in A$. 

We claim that line \ref{line: MJ increment} can execute at most $\frac{4L\norm{c}_\infty}{\eps} + 1$ times for the index $k$. After just one iteration $G^k(\psi) > 0$ and so by Lemma \ref{lem: psi bound} we see that $\psi^k < \psi^j + 2{\norm{c}_\infty}$. 
 Now note that in each iteration of line \ref{line: MJ increment}, $G^k$ increases by at least $\eps$ and so $\psi^k$ decreases by at least $\frac{\eps}{L}$. Hence after iterating line \ref{line: MJ increment},  $\frac{4L\norm{c}_\infty}{\eps}$ more times we will have decreased $\psi^k$ by at least $4\norm{c}_\infty$, and so we will have $\psi^k < \psi^j - {2\norm{c}_\infty}$ which would give us $G^j(\psi) = 0$ by Lemma \ref{lem: psi bound} which is a contradiction since $j \not\in A$. 
 
Since line \ref{line: MJ increment} can be executed at most $\frac{4L\norm{c}_\infty}{\eps}+1$ times for each index in $A$, we conclude that it can in total only be executed at most $\frac{4L\norm{c}_\infty}{\eps} \abs{A} + 1 \leq (N-1) (\frac{4L\norm{c}_\infty}{\eps}+ 1)$ times in total and so the result follows. 
\end{proof}

Next we prove that Algorithm \ref{alg: magic juggle} cannot increase the error. 

\begin{prop}\label{prop: magic juggle error}
	
Let $\psi_{in}$ be the input and $\psi_{out}$ be the output of Algorithm \ref{alg: magic juggle}. Suppose that $F^*$ is differentiable, for all $k \in \{1, \dots, N \}$ we have ${(\nabla F^*)}^k \geq 3\eps$, and that 
\begin{align*}
\frac{\partial F^*}{{\partial \psi^k}}(\psi_1) - \frac{\partial F^*}{{\partial \psi^k}}(\psi_2) \geq \sum_{j \neq k} \abs{\frac{\partial F^*}{\partial \psi^j}(\psi_1) - \frac{\partial F^*}{\partial \psi^j}(\psi_2)},
\end{align*}
for all $k \in \{1, \dots, N \}$.
Then $\norm{\nabla \Phi(\psi_{out})}_1 \leq \norm{\nabla \Phi(\psi_{in})}_1$. 
	
\end{prop}

\begin{proof}
	
We analyze what happens in a single execution of line \ref{line: MJ increment} in Algorithm \ref{alg: magic juggle}. Fix some index $i$ so that $G^i(\psi) \leq \eps$ and so the if statement of line \ref{line: MF if} evaluates to true. Let $\psi_1$ be the value of $\psi$ before an execution of line \ref{line: MJ increment} and $\psi_2$ be the value of $\psi$ after. We have
\begin{align*}
&\norm{\nabla\Phi(\psi_2)}_1 \\
&= \sum_{j} \abs{\nabla F^*(\psi_2)^j - G^j(\psi_2)} \\
&= \sum_{j \neq i} \abs{\nabla F^*(\psi_2)^j - G^j(\psi_2)} + \abs{\nabla F^*(\psi_2)^i - G^i(\psi_2)} \\
&\leq \sum_{j \neq i} \(\abs{\nabla F^*(\psi_1)^j - G^j(\psi_1)} + \abs{\nabla F^*(\psi_2)^j - \nabla F^*(\psi_1)^j} + \abs{G^j(\psi_1) - G^j(\psi_2)}\) + \abs{\nabla F^*(\psi_2)^i - G^i(\psi_2)}. 
\end{align*}
Now since by assumption $\nabla F^*(\psi_2)^i \geq 3\eps$ we have $\abs{\nabla F^*(\psi_2)^i - G^i(\psi_2)} = \nabla F^*(\psi_2)^i - G^i(\psi_2)$. Also by monotonicity of $G$, $\abs{G^j(\psi_1) - G^j(\psi_2)} = G^j(\psi_1) - G^j(\psi_2)$ for $j \neq i$. Hence continuing from above we get
\begin{align*}
& \norm{\nabla\Phi(\psi_2)}_1 \\
&\leq \sum_{j \neq i} \(\abs{\nabla F^*(\psi_1)^j - G^j(\psi_1)} + \abs{\nabla F^*(\psi_2)^j - \nabla F^*(\psi_1)^j} + G^j(\psi_1) - G^j(\psi_2)\) + \nabla F^*(\psi_2)^i - G^i(\psi_2) \\
&= \sum_{j \neq i} \(\abs{\nabla F^*(\psi_1)^j - G^j(\psi_1)} + \abs{\nabla F^*(\psi_2)^j - \nabla F^*(\psi_1)^j}\) - G^i(\psi_1) + \nabla F^*(\psi_2)^i \\
&= \norm{\nabla\Phi(\psi_1)}_1 -  \abs{\nabla F^*(\psi_1)^i - G^i(\psi_1)} + \sum_{j \neq i} \abs{\nabla F^*(\psi_2)^j - \nabla F^*(\psi_1)^j} - G^i(\psi_1) + \nabla F^*(\psi_2)^i \\
&= \norm{\nabla\Phi(\psi_1)}_1 + \sum_{j \neq i} \abs{\nabla F^*(\psi_2)^j - \nabla F^*(\psi_1)^j} - (\nabla F^*(\psi_1)^i - \nabla F^*(\psi_2)^i),
\end{align*}
where in the last line we have used that $\abs{\nabla F^*(\psi_1)^i - G^i(\psi_1)} = {\nabla F^*(\psi_1)^i - G^i(\psi_1)}$, since $\nabla F^*(\psi_1)^i \geq 3\eps > \eps \geq  G^i(\psi_1)$. Now by assumption we have
\begin{align*}
\frac{\partial F^*}{{\partial \psi^i}}(\psi_1) - \frac{\partial F^*}{{\partial \psi^i}}(\psi_2) \geq \sum_{j \neq i} \abs{\frac{\partial F^*}{\partial \psi^j}(\psi_1) - \frac{\partial F^*}{\partial \psi^j}(\psi_2)},
\end{align*}
from which it follows that $\sum_{j \neq i} \abs{\nabla F^*(\psi_2)^j - \nabla F^*(\psi_1)^j} + \nabla F^*(\psi_2)^i - \nabla F^*(\psi_1)^i \leq 0$ and so we get $\norm{\nabla\Phi(\psi_2)}_1 \leq \norm{\nabla\Phi(\psi_1)}_1$ and so the error does not increase in any execution of line \ref{line: MJ increment}. Since line \ref{line: MJ increment} is the only line of Algorithm \ref{alg: magic juggle} that changes $\psi$ we conclude that $\norm{\nabla \Phi(\psi_{out})}_1 \leq \norm{\nabla \Phi(\psi_{in})}_1$ as desired.
\end{proof}

\begin{rmk}
	
	We remark that if $F$ is twice differentiable then the condition
	\begin{align*}
	\frac{\partial F^*}{{\partial \psi^k}}(\psi_1) - \frac{\partial F^*}{{\partial \psi^k}}(\psi_2) \geq \sum_{j \neq k} \abs{\frac{\partial F^*}{\partial \psi^j}(\psi_1) - \frac{\partial F^*}{\partial \psi^j}(\psi_2)},
	\end{align*}
	that is assumed in the above proposition is equivalent to
	\begin{align*}
	\frac{\partial^2 F^*}{{(\partial \psi^k)}^2} \geq \sum_{j \neq k} \abs{\frac{\partial^2 F^*}{\partial \psi^k \partial \psi^j}},
	\end{align*}
	which says that the Hessian of $F^*$ is diagonally dominated. 
	
	To see this note that the forward direction follows from dividing by $\norm{\psi_1 - \psi_2}$ and limiting as $\psi_2$ approaches $\psi_1$. The backward direction follows from taking the line integral from $\psi_2$ to $\psi_1$. 
\end{rmk}

\section{Convergence of Newton Algorithm}\label{sec: Convergence of Newton Algorithm}

In this section we will prove convergence of  Algorithm \ref{alg: damped newton}. We will show convergence under weaker assumptions then that of Theorem \ref{thm: Newton Convergence}, however in this section the assumptions are stated in terms of $F^*$ instead of $F$. We shall wait till Section \ref{sec: Relationship between F and F*} to find conditions on $F$ itself, that give convergence of our algorithm. 

We recall \cite[Theorem 4.1]{KitagawaMerigotThibert19} which says that under our assumptions $G$ is $C^{1, \alpha}$ on $\mathcal{K^\eps}$ for every $\eps > 0$.  

\begin{prop}\label{prop: Newton Convergence}

Let $F$ be a strictly convex storage fee function. Suppose that there is some $\eps > 0$ so that ${(\nabla F^*(\psi))}^i \geq \eps$ for every $\psi \in \mathcal{K}^0$. Set $\eps_0 = \frac{\eps}{4}$. 
	
Furthermore assume that $\nabla F^*$ is $C^{1, \alpha}$ on $\mathcal{K}^{\eps_0}$. We let $L$ be the sum of the $C^{1, \alpha}$ constants of $\nabla F^*$ and $G$ on $\mathcal{K}^{\eps_0}$.

Next assume that $\Phi$ is strongly concave (except in the direction $\onevect$) for all $\psi \in \mathcal{K}^{\eps_0}$. More formally, there is some $\kappa > 0$ so that for all $\psi \in \mathcal{K}^{\eps_0}$
\begin{align*}
D^2 \Phi(\psi) \leq -\kappa H
\end{align*}
where $H$ is the orthogonal projection onto the hyperplane perpendicular to $\onevect$.

Finally suppose that 
\begin{align*}
\frac{\partial^2 F^*}{{(\partial \psi^i)}^2} \geq \sum_{j \neq i} \abs{\frac{\partial^2 F^*}{\partial \psi^i \partial \psi^j}}
\end{align*}
on $\mathcal{K}^0$. 

Then the iterates of Algorithm \ref{alg: damped newton} satisfy 
\begin{align*}
\norm{\nabla  \Phi(\psi_{k+1})}_1 \leq (1 - \frac{\conj \tau_k}{2}) \norm{\nabla \Phi(\psi_{k})}_1
\end{align*}
where
\begin{align*}
\conj \tau_k = \min(\frac{\kappa^{1+\frac{1}{\alpha}}\eps_0}{L^{\frac{1}{\alpha}} \norm{\gtp(\psi)}2^{\frac{1}{\alpha}}} , 1).
\end{align*}
Furthermore once we have $\conj \tau_k =1$ we get
\begin{align*}
\norm{\nabla  \Phi(\psi_{k+1})} \leq \frac{L\norm{\gtp(\psi)}^{1 + \alpha}}{\kappa^{1+\alpha}}.
\end{align*}
\end{prop}

\begin{proof}

We analyze a single iteration of Algorithm \ref{alg: damped newton}. Define $\psi := \psi_k \in \mathcal{K}^{2\eps_0}$. 

Let $v:= [D^2  \Phi(\psi)]^+ (\nabla  \Phi(\psi))$ where $[D^2  \Phi(\psi)]^+$ is the pseudo-inverse of $D^2  \Phi(\psi)$. We see that
\begin{align*}
\norm{v} \leq \frac{\norm{\nabla  \Phi(\psi)}} {\kappa},
\end{align*}
as $\Phi$ was $\kappa$-concave except in the direction of $\onevect$ and $\inner{\nabla \Phi}{\onevect} = 0$.

Also define $\psi_\tau = \psi - \tau v$. Let $\tau_1$ be the first exit time from $\mathcal{K}^{\eps_0}$, i.e. $\min_i G(\psi_{\tau_1})^i = \eps_0$ and $\tau_1$ is the smallest value of $\tau$ for which this holds. We have that
\begin{align*}
\eps_0 \leq \norm{G(\psi_{\tau_1}) - G(\psi)} \leq L \tau_1 \norm{v} \leq \frac{L}{\kappa}\norm{\nabla  \Phi(\psi)}\tau_1
\end{align*}
and so
\begin{align*}
\tau_1 \geq \frac{\kappa\eps_0}{L\norm{\gtp(\psi)}}.
\end{align*}
Applying Taylor's formula to $\gtp$ we get
\begin{equation}\label{eqn: ti Phi Taylor}
\gtp(\psi_\tau) 
= \gtp(\psi-\tau v) 
= \gtp(\psi) - \tau (D\gtp(\psi))v + R(\tau)
= \gtp(\psi) - \tau \gtp(\psi) + R(\tau)
\end{equation}

where
\begin{align*}
\norm{R(\tau)}
&= \norm{ \int_0^\tau (D\gtp(\psi_\sigma) - D\gtp(\psi))v  d\sigma } \\
&\leq \int_0^\tau L \norm{\psi_\sigma - \psi}^\alpha \norm{v} d\sigma\\
&= \int_0^\tau L \norm{\sigma v}^\alpha \norm{v} d\sigma\\
&= L \norm{v}^{\alpha+1} \int_0^\tau {\sigma}^\alpha d\sigma\\
&= L \norm{v}^{\alpha+1} \frac{\tau^{\alpha + 1}}{\alpha+1}\\
&\leq \frac{L\norm{\gtp(\psi)}^{1 + \alpha}}{\kappa^{1+\alpha}}\tau^{1+\alpha} \numberthis \label{eqn: remainder estimate}
\end{align*}
for $\tau \leq 1$. 

Now we establish the error reduction estimates. \eqref{eqn: ti Phi Taylor} gives
\begin{align*}
\gtp(\psi_\tau) 
= (1-\tau)\gtp(\psi) + R(\tau)
\end{align*}
so we have
\begin{align*}
\norm{\gtp(\psi_\tau)}_1 \leq (1-\frac{\tau}{2})\norm{\gtp(\psi)}_1
\end{align*}
provided that
 \begin{align*}
\norm{R(\tau)}_1 \leq \frac{\tau}{2} \norm{\gtp(\psi)}_1.
\end{align*}
Since $\norm {R(\tau)}_1 \leq \sqrt{N}\norm{R(\tau)}$ and $\norm{\gtp(\psi)} \leq \norm{\gtp(\psi)}_1$ we just need
\begin{align*}
\norm{R(\tau)} \leq \frac{\tau}{2\sqrt{N}} \norm{\gtp(\psi)}
\end{align*}
Again using \eqref{eqn: remainder estimate} this will be true provided
\begin{align*}
\tau \leq \min(1,\tau_1, \frac{\kappa^{1+\frac{1}{\alpha}}}{L^{\frac{1}{\alpha}} \norm{\gtp(\psi)}(2\sqrt{N})^{\frac{1}{\alpha}}}) =: \tau_2.
\end{align*}
Hence we see that if we set $\conj \tau_k := \tau_2$, then the claim is true. Furthermore as the error goes to zero, eventually we must have $\conj \tau_k = 1$. When this happens \eqref{eqn: ti Phi Taylor} gives
\begin{align*}
\gtp(\psi_1) = R(1) 
\end{align*}
and so \eqref{eqn: remainder estimate} gives the super-linear convergence. 
\end{proof}

\begin{rmk}
There are two ways to satisfy the strong concavity assumption on $\Phi$. Either one can assume a PW-inequality in which case we get the strong concavity from the $G$ term. Alternatively one can put some kind of strong concavity assumption on $F^*$. We will see that in the case where $F$ splits into $f_i$ this strong concavity assumption on $F^*$ will be satisfied assuming some regularity and convexity conditions on the $f_i$. 
	
\end{rmk}

\section{Relationship between $F$ and $F^*$}\label{sec: Relationship between F and F*}

This section is mainly about the convex analysis of storage fee functions that split. First we show that the assumption that $F$ splits yields the technical condition on $F^*$ that we needed in Proposition \ref{prop: magic juggle error} in order to obtain that Algorithm \ref{alg: magic juggle} does not increase error. Next we analyze how the regularity and convexity assumptions on the $f_i$ effect the regularity of $F^*$. We then use this to prove our main convergence theorem, Theorem \ref{thm: Newton Convergence}.

We recall some notation and definitions from convex analysis. We start with two definitions from \cite[Section 26]{Rockafellar70}.
\begin{defin}
Given any proper convex function $G$ we say that $G$ is essentially smooth if the following holds. Let $C = \interior(\dom G)$. We require that $C \neq \emptyset$ and $G$ is differentiable on $C$. Also if $x_i \in C$ is a sequence that converges to a point on the boundary of $C$ then we require that ${\abs{\nabla G(x_i)}}$ diverges to $+\infty$. 
\end{defin}

\begin{defin}
Given any proper convex function $G$ we say that $G$ is essentially strictly convex if $G$ is strictly convex on every convex subset of $\{x : \partial G(x) \neq \emptyset \}$.
\end{defin}

Next we recall that $\range \partial G = \bigcup_{x \in \R^n} \partial G(x)$. Finally we have from \cite[Corollary 23.5.1]{Rockafellar70} that $\range \partial G^* \subset \dom G$ (see page 227). 

\begin{lem}\label{lem: F*range}
	
	Suppose that $F$ is an essentially strictly convex storage fee function. Then $F^*$ is an everywhere finite and differentiable convex function with $\range \nabla F^* \subset \wvs$. 
\end{lem}

\begin{proof}
	It is obvious that $F^*$ is finite everywhere, as $\dom F$ is compact and so $\sup_{x \in \dom F} (\inner xy - F(x))$ is always finite. 
	
	Since $F$ is essentially strictly convex we have that $F^*$ is essentially smooth by \cite[Theorem 26.3]{Rockafellar70}. 
	Since we have seen that $F^*$ is finite everywhere, \cite[Theorem 26.1]{Rockafellar70} tells us that $F^*$ is differentiable everywhere. 
	
	For the last claim we have $\range \nabla F^* \subset \dom F \subset \wvs$ by \cite[Corollary 23.5.1]{Rockafellar70} (see page 227). 
\end{proof}

Next we obtain a characterization of the subdifferential of $F^*$ purely in terms of $F$, when $F$ is a storage fee function that splits. This characterization will form the basis of our program to translate conditions on $F$ into conditions on $F^*$.  

\begin{lem}\label{lem: subdiff spliting}
	Suppose that the storage fee function $F$ splits and each $f_i$ is strictly convex. Furthermore, assume that $F$ is not the indicator function of a point. Then the system
	\begin{align}
	\psi^i &\in \partial f_i(\wv^i) + r \label{eqn: condition 1}\\
	\sum_i \wv^i &= 1 \label{eqn: condition 2}
	\end{align}	
	characterizes the subdifferential of $F^*$ in the sense that for any pair $(\psi, \wv)$ there exists an $r \in \R$ so that the above system is satisfied if and only $\wv = \nabla F^*(\psi)$. 
\end{lem}

\begin{proof}	
	
	Let $\pi_i: \R^N \to \R$ denote the projection onto the $i$-th coordinate. For this proof we shall use the notation $F_i = f_i \circ \pi_i$. Note that $\partial F_i(\wv) = (\partial f_i(\wv^i)) e_i$ where $e_i$ is the $i$-th standard coordinate (to clarify $(\partial f_i(\wv^i)) e_i = \{x e_i : x \in \partial f_i(\wv^i) \}$). 
	
	We also remark that since each $f_i$ is strictly convex, $F$ is strictly convex and so $F^*$ is differentiable hence it makes sense to refer to $\nabla F^*$. 
	
	Note that since $f_i(v) = +\infty$ if $v \not\in [0,1]$ we can write	
	\begin{align*}
	F(\wv) = \sum_i F_i(\wv) + \delta_{P}(\wv)
	\end{align*}
	where $P$ is the hyperplane that extends $\wvs$ (formally $P = \{\wv \in \R^N: \sum_i \wv^i = 1 \}$). Our next objective is to show that the intersection of the relative interiors of the domains of the $F_i$ and $P$ is non-empty (or the $f_i$ can be modified to make this so). 
	
	Note that since the $f_i$ are proper, and convex their domains are non-empty intervals, say $\conj{\dom f_i} = [a^i, b^i]$. Note that we must have	
	\begin{align*}
	\sum_i a^i < 1 < \sum_i b^i
	\end{align*}
	since, if $\sum_i a^i = 1$, then $F = \delta_{\{a \}}$ where $a = (a^1, \dots, a^N)$ and so $F$ would be the indicator function of a point (a similar argument holds if $\sum_i b^i = 1$). 
	Now we can choose $\wv \in P$ so that $\wv^i \in (a^i, b^i)$ whenever $a^i < b^i$ and $\wv^i = a^i$ if $a^i = b^i$. In this case $\wv^i \in \ri \dom f_i$ which implies that $\wv \in \ri \dom F_i$. Furthermore $\wv \in P = \ri \dom \delta_{P}$.

	To recount we have proven that $\cap_i \ri \dom F_i \cap \ri \dom \delta_P \neq \emptyset$. Hence by \cite[Theorem 23.8]{Rockafellar70} we have that $\partial F = \partial F_1 + \dots + \partial F_N + \partial \delta_{P}$.
	
	We can now proceed to the proof of the lemma: for any pair $(\psi, \wv)$ there exists an $r \in \R$ so that the above system is satisfied if and only $\wv = \nabla F^*(\psi)$. 
	
	For the backward direction, by Lemma \ref{lem: F*range} we have $\wv \in P$ and so $\wv$ satisfies \eqref{eqn: condition 2}. However since $\wv \in P$ we have that $\partial \delta_{P}(\wv) = \{ r \onevect : r \in \R \}$, as $P$ is a plane orthogonal to $\onevect$. Hence \eqref{eqn: condition 1} simply says that $\psi \in \partial F_1(\wv) + \dots + \partial F_N(\wv) + \partial \delta_{P}(\wv) = \partial F(\wv)$, but this is given, since we assumed $\wv = \nabla F^*(\psi)$.
	
	For the forward direction say that $\psi, \wv, r$ is a solution. Because of \eqref{eqn: condition 2} we have that $\wv \in P$. Then since $\partial F = \partial F_1 + \dots + \partial F_N + \partial \delta_{P}$ we see that \eqref{eqn: condition 1} says that $\psi \in \partial F(\wv)$. Since $F^*$ is everywhere differentiable this means that $\wv = \nabla F^*(\psi)$.	
\end{proof}

Next we use our characterization in order to show that $F$ splitting implies a monotonicity condition on the partial derivatives of $F^*$. 

\begin{prop}\label{prop: splitting}
	
	Suppose that the storage fee function $F$ splits and each $f_i$ is strictly convex. Say that $\psi_1 \in \R^n$ and $\gamma > 0$ are fixed. Then for $j \neq k$
	\begin{align*}
	\frac{\partial F^*}{\partial \psi^j}(\psi_1 + \gamma e_k) \leq \frac{\partial F^*}{\partial \psi^j}(\psi_1). 
	\end{align*}
\end{prop}

\begin{proof}
	
	First, if $F$ is the indicator function of a point, say $F = \delta_{\{a\}}$, then $F^*(\psi) = \inner{\psi}{a}$ and so $\nabla F^*(\psi) = a$. Therefore $\frac{\partial F^*}{\partial \psi^j}(\psi_1 + \gamma e_k) = a^j = \frac{\partial F^*}{\partial \psi^j}(\psi_1)$. Hence we may assume that $F$ is not the indicator function of a point. 
		
	Set $\psi_2 = \psi_1 + \gamma e_k$, $\wv_1 = \nabla F^*(\psi_1)$, and $\wv_2 = \nabla F^*(\psi_2)$. Use Lemma \ref{lem: subdiff spliting} to find $r_1, r_2$ so that $(\psi_1, \wv_1, r_1)$ and $(\psi_2, \wv_2, r_2)$ satisfy the system given by \eqref{eqn: condition 1} and \eqref{eqn: condition 2}.  The objective is to show that for $j \neq k$ we have $\wv_1^j \geq \wv_2^j$. 
	
	If $\wv_2 = \wv_1$ then we are done. If not, there must be some index $l$ so that $\wv_2^l < \wv_1^l$ (as $\sum_i \wv_1^i = \sum_i \wv_2^i = 1$). Note that since $F^*$ is convex, $\nabla F^*$ is monotone and so we must have $\wv_2^k \geq \wv_1^k$. In particular $l \neq k$. 
	
	Since $f_l$ is strictly convex we have that $\wv_2^l < \wv_1^l$ implies that $\partial f_l(\wv_2^l) < \partial f_l(\wv_1^l)$ in the sense that if $x \in \partial f_l(\wv_2^l)$ and $y \in \partial f_l(\wv_1^l)$ then $x < y$. However now \eqref{eqn: condition 1} tells us that $\psi_1^l = \psi_2^l \in (\partial f_l(\wv_1^l) + r_1) \cap (\partial f_l(\wv_2^l) + r_2)$. In particular $(\partial f_l(\wv_1^l) + r_1) \cap \partial (f_l(\wv_2^l) + r_2) \neq \emptyset$. Since $\partial f_l(\wv_2^l) < \partial f_l(\wv_1^l)$ we conclude that $r_1 < r_2$. 
	
	Now for any $j \neq k$ we have $\psi_1^j = \psi_2^j$ and so as above we get $(\partial f_j(\wv_1^j) + r_1) \cap (\partial f_j(\wv_2^j) + r_2) \neq \emptyset$. Since $r_1 < r_2$ and $f_j$ is strictly convex this is only possible if $\wv_2^j \leq \wv_1^j$ as desired. 
\end{proof}

We now use the monotonicity condition on the partial derivatives of $F^*$ to obtain the technical condition used in section \ref{sec: Parameter Shuffling}.

\begin{cor}\label{cor: splitting}
	
	Assume $F$ is a storage fee function that splits with strictly convex $f_i$. Let $\psi_1 \in \R^n$ and set $\psi_2 =\psi_1 - \gamma e_k$, for some $\gamma \geq 0$. Then
	\begin{align*}
	\frac{\partial F^*}{{\partial \psi^k}}(\psi_1) - \frac{\partial F^*}{{\partial \psi^k}}(\psi_2) = \sum_{j \neq k} \abs{\frac{\partial F^*}{\partial \psi^j}(\psi_1) - \frac{\partial F^*}{\partial \psi^j}(\psi_2)}.
	\end{align*}
	In particular if $F^*$ is twice differentiable then
	\begin{align*}
	\frac{\partial^2 F^*}{{(\partial \psi^k)}^2} = \sum_{j \neq k} \abs{\frac{\partial^2 F^*}{\partial \psi^k \partial \psi^j}}
	\end{align*}
\end{cor}

\begin{proof}
	
	By a direct application of Proposition \ref{prop: splitting} we see that for $j \neq k$
	\begin{align*}
	\frac{\partial F^*}{\partial \psi^j}(\psi_1) - \frac{\partial F^*}{\partial \psi^j}(\psi_2) \leq 0.
	\end{align*}
	In particular 
	\begin{align*}
	\frac{\partial F^*}{{\partial \psi^k}}(\psi_1) - \frac{\partial F^*}{{\partial \psi^k}}(\psi_2) - \sum_{j \neq k} \abs{\frac{\partial F^*}{\partial \psi^j}(\psi_1) - \frac{\partial F^*}{\partial \psi^j}(\psi_2)}
	&= \frac{\partial F^*}{{\partial \psi^k}}(\psi_1) - \frac{\partial F^*}{{\partial \psi^k}}(\psi_2) + \sum_{j \neq k} \(\frac{\partial F^*}{\partial \psi^j}(\psi_1) - \frac{\partial F^*}{\partial \psi^j}(\psi_2) \) \\
	&= \sum_{j} \(\frac{\partial F^*}{\partial \psi^j}(\psi_1) - \frac{\partial F^*}{\partial \psi^j}(\psi_2) \) \\
	&= \inner{\nabla F^*(\psi_1)}{\onevect} - \inner{\nabla F^*(\psi_2)}{\onevect} \\
	&= 0
	\end{align*}
	where the last line follows from Lemma \ref{lem: F*range}, as $\range \nabla F^* \subset \wvs$ implies that $\inner{\nabla F^*(\psi_1)}{\onevect} = \inner{\nabla F^*(\psi_2)}{\onevect} = 1$. This proves the first claim. The second follows easily by taking limits. 
\end{proof}

In the next theorem we show that convexity and regularity assumptions on $F$ give higher order regularity on $F^*$ when $F$ is a storage fee function that splits. In the process we also obtain an explicit formula for $D^2F^*$ in terms of the $f_i$. 

\begin{thm}\label{thm: smoothness of F*}
	
	Let $F$ be a storage fee function that splits where each $f_i$ is an essentially smooth function that is twice differentiable on the interior of its domain and so that each $f_i''$ is locally Lipschitz (on the interior of its domain). Furthermore we assume each $f_i$ is strongly convex, say $f_i''(x) \geq \eta$. 
	Then $F^*$ is $C^{2,1}(\mathcal{K}^0)$. 
	
	Finally $\norm{D^2F^*(\psi_1) - D^2F^*(\psi_2)} \leq 4N^{2} C_f \eta^{-3} \norm{\psi_1 - \psi_2}$ where
	\begin{align*}
	C_f = \max_i \sup_{x,y \in S_i } {\frac {\abs{f_i''(x) - f_i''(y)}} {\abs{x-y}}}
	\end{align*}
	where each $S_i$ is a compact subset of the interior of $\dom f_i$ that is constructed in the proof. 
\end{thm}

\begin{proof}
	Since $f_i$ are proper and convex their domains are non-empty intervals, say $\conj{\dom f_i} = [a^i, b^i]$. Since the $f_i$ are essentially smooth they aren't indicator functions of points, and so $b^i > a^i$. Since the $f_i$ are essentially smooth $\partial f_i(a^i) = \partial f_i(b^i) = \emptyset$ (the derivative here is ``$\infty$'').
	
	Since by assumption the $f_i$ are differentiable on $(a_i, b_i)$, Lemma \ref{lem: subdiff spliting} tells us that $\wv \in \partial F^*(\psi)$ if and only if $\wv^i \in (a_i, b_i)$ and there exists $r \in \R$ so that 
	\begin{equation}\label{eqn: system}
	\begin{split}
	\psi^i &= f_i'(\wv^i) + r \\
	\sum_i \wv^i &= 1 .
	\end{split}
	\end{equation}
	
	Since $F$ was strictly convex we have that $F^*$ is continuously differentiable. Hence for any fixed $\psi$ there precisely one $\wv$ that satisfies the above system (it is $\nabla F^*(\psi)$). Hence (looking at the first equation) there is also precisely one value of $r$ that satisfies the system \eqref{eqn: system}. We this denote by $r(\psi)$. 
	
	Our next step is to apply the implicit function theorem to deduce the differentiability of $\nabla F^*(\psi)$. 
	
	Let $H(\psi, \wv, r) = (f_1'(\wv^1) + r - \psi^1, \dots,  f_N'(\wv^N) + r - \psi^N, \sum_i \wv^i - 1 )$ encode the system in the sense that $H(\psi, \wv, r) = 0$ if and only if $(\psi, \wv, r)$ satisfies \eqref{eqn: system}. Let $J$ be the Jacobian matrix of $H$ with respect to $(\wv, r)$. If we can show that $J$ is invertible then by the implicit function theorem $\nabla F^*(\psi)$ and $r(\psi)$ will be continuously differentiable. Furthermore by the implicit function theorem we will have
	\begin{align*}
	\pat{\psi_j}(\nabla F^*) 
	= - J\i \frac{\partial H}{\partial \psi_j}
	= - J\i (-e_j)
	= J\i e_j
	\end{align*}
	and so $D^2 F^*$ is just $J\i$ with the last row and column removed (these correspond to the $r$ terms). 
	
	We proceed to compute $J$ and its inverse. Direct computation shows that
	\begin{align*}
	J = \begin{bmatrix}
	f_1''(\wv^1)& &  &1 \\
	& \ddots & & \vdots \\
	&  & f_N''(\wv^N) & 1\\
	1 & \cdots &1 &0
	\end{bmatrix}
	\end{align*}
	where empty entries are interpreted to be zero. Let $l^i(\wv) = f_i''(\wv^i)\i$. It is now easy to see that $J\i = L + R$ where $L$ is the diagonal matrix with diagonal entries $l^1, l^2, \dots, l^N, 0$ and 
	\begin{align*}
	R = -Q\begin{bmatrix}
	(l^1)^2 & l^1 l^2 &\cdots & l^1 l^N & -l^1 \\
	l^1 l^2 &(l^2)^2&\cdots & l^2 l^N & -l^2\\
	\vdots & \vdots & \ddots & \vdots & \vdots\\
	l^1 l^N &l^2 l^N&\cdots &(l^N)^2 & -l^N \\
	-l^1 & -l^2 & \cdots & -l^N & 1
	\end{bmatrix}
	\end{align*}
	where $Q = (\sum_i l^i)\i$. Hence $J$ was invertible after all. Note that we needed $f_i''(\wv^i) > 0$. 
	
	At this point we have proved that $F^*$ is $C^2$ and $r(\psi)$ is $C^1$. The only thing left to do is to obtain the Lipschitz estimate on $D^2F^*$. However (as seen in the expression for $J\i$) this strongly relies on the local Lipschitz constant for $f_i''$. Since $f_i''$ is only locally Lipschitz we need to obtain control of the possible values of $\nabla F^*(\psi)$ and assure that they never approach the boundary of the domain of any $f_i$. Thankfully as we will see in Lemma \ref{lem: psi bound} the assumption that no cell collapses ($\psi \in \mathcal{K}^0$) gives us enough compactness in the $\psi$'s.

	Now fix some $\psi \in \mathcal{K}^0$ and let $\wv = \nabla F^*(\psi)$ and $r = r(\psi)$ be the solutions to our system, \eqref{eqn: system}. Set $\ti \psi = \psi - r \ov$. Define $B$ by
	\begin{align*}
	\max_i (b_i - \wv^i) 
	\geq \frac{\sum_i (b_i - \wv^i)}{N}
	= \frac{(\sum_i b_i) - 1}{N}
	=: B.
	\end{align*}
	Now since $F$ is not the indicator function of a point, we have that $\sum_i b_i > 1$ and so $B > 0$.
	Pick $k$ so that $b_k - \wv^k \geq B$ which implies $\wv^k \leq b_k - B$. Then $\ti \psi^k = f_k'(\wv^k) \leq f_k'(b_k-B)$ where we have used the monotonicity of $f_k'$ (recall the $f_k$ are convex). Since $\ti \psi$ and $\psi$ differ by a multiple of $\ov$ we have $G(\ti \psi) = G(\psi)$ and so $\ti \psi \in \mathcal K^0$. Hence by Lemma \ref{lem: psi bound} we get $\ti \psi^i \leq f_k'(b_k-B) + 2 \norm{c}_\infty$ for all $i$. Setting $A := \frac{1 - \sum_i a_i}{N}$, a symmetric argument gives that $\ti \psi^i \geq f_k'(a_k+A) - 2 \norm{c}_\infty$ for all $i$. In particular $\abs{\ti \psi^i} \leq C_1$ where $C_1 = \max_k(-f_k'(a_k+A), f_k'(b_k-B) ) + 2 \norm{c}_\infty$ is a constant that depends only on $c$ and the $f_k$.
	
	Now, recall from the system \eqref{eqn: system} that $\ti \psi^i = f_i'(\wv^i)$. Hence we have obtained that for any $\psi \in \mathcal K^0$, $ \abs{f_i'(\nabla F^*(\psi))} \leq C_1$. Since  $f_i''$ is locally Lipschitz and essentially smooth there is a constant $C_{f_i}$ that depends only on $f_i$ and $C_1$ so that $\abs{f_i''(x) - f_i''(y)} \leq C_{f_i}\abs{x-y}$ whenever $\abs{f_i'(x)}, \abs{f_i'(y)} \leq C_1$. With this we finally get that for any $\psi_1, \psi_2 \in \mathcal K^0$,
	\begin{align*}
	\abs{f_i''(\nabla F^*(\psi_1)^i) - f_i''(\nabla F^*(\psi_2)^i)} \leq C_{f_i}\abs{\nabla F^*(\psi_1)^i - \nabla F^*(\psi_2)^i}
	\end{align*}
	as desired. We may now return to the problem of controlling $D^2 F^*$. 
	
	Since $D^2 F^*$ is just $J\i$ with the last row and column punctured we get, $D^2 F^* = \S + \T$ where $\S, \T$ arise from $L, R$ respectively by removing the last row and column. 
	We are now in a position to obtain the $C^{0,1}$ bound on $D^2F^*$. Fix some $\psi_1, \psi_2 \in \mathcal{K}^0$. Let $\wv_1, \wv_2$ equal $\nabla F^*(\psi_1), \nabla F^*(\psi_2)$ respectively. For $i \in \{1,2\}$ we write $\S_i, \T_i, Q_i, l_i^j$ for the $\S,\T,Q, l^j$ corresponding to $\wv_1, \wv_2$. 
	
	First note that since the $f_i$ are strongly convex with parameter $\eta$, $F$ is strongly convex with parameter $\eta$. Hence $\nabla F^*$ is Lipschitz with constant $\frac{1}{\eta}$. Hence 
	\begin{align*}
	\norm{\wv_1 - \wv_2} 
	= \norm{\nabla F^*(\psi_1) - \nabla F^*(\psi_2)}
	\leq \frac{1}{\eta} \norm{\psi_1 - \psi_2}
	\end{align*}
	and so $\abs{f_i''(\wv_1^i) - f_i''(\wv_2^i)} \leq C_f \eta\i \norm{\psi_1 - \psi_2}$ where $C_f$ is the maximum of the $C_{f_i}$ constants. In particular
	\begin{align*}
	\abs {l_1^i - l_2^i} 
	= \abs {f_i''(\wv_1^i)\i - f_i''(\wv_2^i)\i}
	= \abs { \frac{f_i''(\wv_1^i) - f_i''(\wv_2^i)}{f_i''(\wv_1^i)f_i''(\wv_2^i)} } 
	\leq \eta^{-2} (C_f \eta\i \norm{\psi_1 - \psi_2})
	= C_f \eta^{-3} \norm{\psi_1 - \psi_2}
	\end{align*}
	Now $\norm{D^2F^*(\psi_1) - D^2F^*(\psi_2)} \leq \norm{\S_1-\S_2} + \norm{\T_1 - \T_2}$. The $\S$'s are easy to bound: 
	\begin{align*}
	\norm{\S_1-\S_2} = (\sum_i \abs {l_1^i - l_2^i}^2)^{1/2} \leq C_f \sqrt{N} \eta^{-3} \norm{\psi_1 - \psi_2}.
	\end{align*} 
	For the $\T$'s, consider the functions $g_{ij}: \R^n \to \R, l \mapsto \T_{ij}(l) = Q(l) l^il^j = l^il^j(\sum_k l^k)\i $. We see that $g_{ij}$ is continuously differentiable outside the origin and
	\begin{align*}
	\frac{\pa g_{ij}}{\pa (l^m)} = -l^il^j(\sum_k l^k)^{-2} + (\delta_{mi} l^j + \delta_{mj} l^i)(\sum_k l^k)^{-1}
	\end{align*} 
	In particular since $l^i(\sum_k l^k)^{-1} \leq 1$ and $l^j(\sum_k l^k)^{-1} \leq 1$ we have that $\abs{\frac{\pa g_{ij}}{\pa (l^m)}} \leq 3$ and so $g_{ij}$ is Lipschitz with constant $3\sqrt{N}$. Hence we get
	\begin{align*}
	\norm{\T_1 - \T_2} 
	= \(\sum_{i,j} (g_{ij}(l_1)- g_{ij}(l_2))^2 \)^{1/2}
	\leq 3\sqrt{N}\(\sum_{i,j} \norm{l_1- l_2}^2 \)^{1/2}
	\leq 3N^{2} C_f \eta^{-3} \norm{\psi_1 - \psi_2} 
	\end{align*}
	Putting it together we get $\norm{D^2F^*(\psi_1) - D^2F^*(\psi_2)} \leq 4N^{2} C_f \eta^{-3} \norm{\psi_1 - \psi_2}$. 
\end{proof}

Using our explicit expression for $D^2F^*$, we prove a quick corollary which gives invertibility of $D^2F^*$ except in the direction of $\onevect$.

\begin{cor}\label{cor: splitting invertibility}
Let $F$ satisfy the assumptions of Theorem \ref{thm: smoothness of F*}. Then for all $\psi \in \R^N$ we have that $\ker D^2F^*(\psi) = \spn{\onevect}$.
\end{cor}

\begin{proof}
	
Fix some $\psi \in \R^N$ and let $v \in \R^N$ be some vector where $v = \sum_i v^i e_i$ for some $v_i \in \R$. We recall from the proof of Theorem \ref{thm: smoothness of F*} that we can split $D^2F^*(\psi)$ into the sum of two matrices $S,T$ where $S$ is the diagonal matrix with elements $l^1, l^2, \dots, l^N$ and
\begin{align*}
T = -(\sum_i l^i)\i\begin{bmatrix}
(l^1)^2 & l^1 l^2 &\cdots & l^1 l^N  \\
l^1 l^2 &(l^2)^2&\cdots & l^2 l^N \\
\vdots & \vdots & \ddots & \vdots \\
l^1 l^N &l^2 l^N&\cdots &(l^N)^2 
\end{bmatrix}
\end{align*}
where $l^i = f_i'' (\nabla F^*(\psi)^i)\i > 0$. Hence
\begin{align*}
((S+T)v)^j = l^j v^j -  \frac{\sum_k l^jl^k v^k}{\sum_i l^i} = l^j \(v^j - \frac{\sum_k l^k v^k}{\sum_i l^i} \).
\end{align*}
Since each $l^j \neq 0$ we see that $v \in \ker D^2F^*(\psi)$ if and only if for all $j \in \{1, \dots, N\}$ we have that $v^j = \frac{\sum_k l^k v^k}{\sum_i l^i}$ which occurs if and only if all of the $v^j$ are equal, i.e. $v \in \spn{\onevect}$.
\end{proof}

Finally we apply Theorem \ref{thm: smoothness of F*} to prove our main convergence result. 

\begin{proof}[Proof of Theorem \ref{thm: Newton Convergence}]
	
We need to verify all of the conditions of Proposition \ref{prop: Newton Convergence} are satisfied. First note that by Theorem \ref{thm: smoothness of F*} we have that $F^* \in C^{2,1}(\mathcal{K}^0)$. 

Next since $\dom f_i \subset [\eps, 1]$ we see that $\dom F \subset [\eps, 1]^N$. Hence by \cite[Corollary 23.5.1]{Rockafellar70} we have that $\range \partial F^* \subset [\eps, 1]^N$. In particular $(\nabla F^*(\psi))^i \geq \eps$ for $\psi \in \mathcal{K}^0$. 

Next by Corollary \ref{cor: splitting invertibility} for every $\psi \in \R^N$ there is some $\kappa(\psi) > 0$ so that $D^2 F^*(\psi) \geq \kappa(\psi) H$ where $H$ is the orthogonal projection onto the hyperplane perpendicular to $\onevect$ (denoted $\ti P$). By Lemma \ref{lem: psi bound} we have that $\conj {\mathcal{K}^0} \cap \ti P$ is a compact set so we can choose a uniform $\kappa > 0$ so that $D^2 F^*(\psi) \geq \kappa H$ for all $\psi \in \conj {\mathcal{K}^0} \cap \ti P$. Since $D^2 F^*(\psi) = D^2 F^*(H(\psi))$ we get $D^2 F^*(\psi) \geq \kappa H$ for all $\psi \in \conj {\mathcal{K}^0}$.

Finally since $F$ splits by Corollary \ref{cor: splitting} we have that
\begin{align*}
\frac{\partial^2 F^*}{{(\partial \psi^i)}^2} = \sum_{j \neq i} \abs{\frac{\partial^2 F^*}{\partial \psi^i \partial \psi^j}}
\end{align*}
on all of $\R^N$. Hence all of the assumptions of Proposition \ref{prop: Newton Convergence} are indeed satisfied.	
\end{proof}

\section{Stability}\label{sec: Stability}

We now begin working towards the proof of our second main theorem, Theorem \ref{thm: regularize}. In this section we will prove that the optimizing weight vector is stable under perturbations in the storage fee functions. We obtain results both in terms of $L^\infty$ perturbations and in terms of perturbations of the domain. 

To start we set some notation. Let
\begin{align*}
\mathcal{C}(\ti\weightvect) = \min_{S_\#\mu =\sum_{i=1}^N \tilde\weightvect^i \delta_{y_i}} \int c(x, S(x)) d\mu = \sup_{\psi\in \R^N} \(-\int \psi^{c^*} d\mu - \inner{\psi}{\ti \weightvect}\).
\end{align*}
The equality between the minimization and maximization in the above definition is just the classical Kantorovich duality and so the supremum is actually obtained (see \cite[Theorem 5.10]{Villani09}). 
Furthermore, for any storage fee function $F$ let $\mathcal{O}_F(\ti\weightvect) = \mathcal{C}(\ti \wv) + F(\ti \wv)$. 

\begin{lem}\label{lem: strong convex}

Let $F_1$ be a storage fee function and $\wv_1$ be a minimizer in the associated problem. Then for any $\wv \in \R^N$
\begin{align*}
\norm{\wv - \wv_1}^2 \leq {8C_LN}(\mathcal{O}_{F_1}(\wv) - \mathcal{O}_{F_1}(\wv_1))
\end{align*}
where $C_L$ is the constant described in \cite[Lemma A.1]{BansilKitagawa20b}. 

\end{lem}

\begin{proof}
	
Recall from \cite[Lemma A.1]{BansilKitagawa20b} that $\mathcal C$ is strongly convex with constant $\frac{1}{4C_L N}$. Since $F_1$ is convex, $\mathcal{O}_{F_1}$ is also strongly convex with the same constant as $\mathcal C$. The result now follows from \cite[Corollary 3.2.3]{Nesterov2018}. 	
\end{proof}

We now use the strong convexity of $\mathcal{C}$ to obtain stability under $L^\infty$ perturbations of $F$. 

\begin{prop}\label{prop: uniform stab}

Let $F_1, F_2$ be storage fee functions such that $\dom(F_1) = \dom(F_2)$. Let $\wv_1, \wv_2$ be the minimizers of problems associated to $F_1, F_2$ respectively. Then $\norm{\wv_1- \wv_2} \leq 4 \sqrt{C_L N \norm{F_1-F_2}_\infty}$.

\end{prop}

\begin{proof}

We have
\begin{align*}
\mathcal{O}_{F_1}(\wv_2) 
&= \mathcal{C}(\wv_2) + F_1(\wv_2) \\
&= \mathcal{O}_{F_2}(\wv_2) + F_1(\wv_2) - F_2(\wv_2) \\
&\leq \mathcal{O}_{F_2}(\wv_1) + F_1(\wv_2) - F_2(\wv_2) \\
&= \mathcal{O}_{F_1}(\wv_1) + F_2(\wv_1) - F_1(\wv_1) + F_1(\wv_2) - F_2(\wv_2).
\end{align*}
Hence $\mathcal{O}_{F_1}(\wv_2) - \mathcal{O}_{F_1}(\wv_1) \leq 2\norm{F_1-F_2}_\infty$ and so the result follows from the Lemma \ref{lem: strong convex}.  
\end{proof}

Finally we show stability under changing the domain of $F$. In order to quantitatively measure perturbations of the domain of $F$, we use Hausdorff distance which we denote with $d_\H$. 

\begin{prop}\label{prop: hausdorff stab}
	
	Let $F_1, F_2$ be storage fee functions such that $\dom(F_1) \subset \dom(F_2)$ and $F_1 = F_2$ on $\dom(F_1)$. Furthermore assume that $F_2$ is uniformly continuous on its domain with modulus of continuity $\omega$. Let $\wv_1, \wv_2$ be the minimizers of problems associated to $F_1, F_2$ respectively. Then	
	\begin{align*}
	\norm{\wv_1- \wv_2}^2 \leq 8C_LN (2 \norm{c}_\infty \sqrt{N} d_\H(\dom(F_2) , \dom(F_1)) +  \omega(d_\H(\dom(F_2) , \dom(F_1)))).
	\end{align*} 
\end{prop}

\begin{proof}
	
Choose $\ti \wv_2 \in \dom F_1$ so that $\norm{\wv_2- \ti \wv_2} \leq d_\H(\dom(F_2) , \dom(F_1))$. Since $F_2$ is uniformly continuous we have $F_2(\ti \wv_2) \leq F_2(\wv_2) + \omega(d_\H(\dom(F_2) , \dom(F_1)))$. 

Now let $\ti \psi_2$ be a maximizer in the dual problem for $\mathcal{C}(\ti \wv_2)$ so that $\ti \psi_2 \in \conj {\mathcal{K}^0}$ and $\sum_j \ti \psi_2^j = 0$. We see that
\begin{align*}
\mathcal{C}(\ti \wv_2)
&= \int (\ti \psi_2)^{c^*} \dmu + \inner{\ti \psi_2}{\ti \wv_2} \\
&= \int (\ti \psi_2)^{c^*} \dmu + \inner{\ti \psi_2}{\wv_2} + \inner{\ti \psi_2}{\ti \wv_2 - \wv_2} \\
&\leq \sup_{\psi} \(\int \psi^{c^*} \dmu + \inner{\psi}{\wv_2} \) + \inner{\ti \psi_2}{\ti \wv_2 - \wv_2} \\
&= \mathcal{C}(\wv_2) + \inner{\ti \psi_2}{\ti \wv_2 - \wv_2} \\
&\leq \mathcal{C}(\wv_2) + \norm{\ti \psi_2} \cdot \norm{\ti \wv_2 - \wv_2} \\
&\leq \mathcal{C}(\wv_2) + 2 \norm{c}_\infty \sqrt{N} d_\H(\dom(F_2) , \dom(F_1)),
\end{align*}
where the final inequality follows from Lemma \ref{lem: psi bound}. Hence	
\begin{align*}
O_{F_2}(\ti \wv_2)
&= \mathcal{C}(\ti \wv_2) + F_2(\ti \wv_2) \\
&\leq \mathcal{C}(\wv_2) + 2 \norm{c}_\infty \sqrt{N} d_\H(\dom(F_2) , \dom(F_1)) + F_2(\wv_2) + \omega(d_\H(\dom(F_2) , \dom(F_1))) \\
&= O_{F_2}(\wv_2) + 2 \norm{c}_\infty \sqrt{N} d_\H(\dom(F_2) , \dom(F_1)) +  \omega(d_\H(\dom(F_2) , \dom(F_1))).
\end{align*}
Next since $\ti \wv_2 \in \dom F_1$ we have $F_1(\ti \wv_2) = F_2(\ti \wv_2)$ and so $O_{F_2}(\ti \wv_2) = O_{F_1}(\ti \wv_2)$. But by the definition of $\wv_1$ we have $O_{F_1}(\ti \wv_2) \geq O_{F_1}(\wv_1)$. Next note that pointwise $F_1 \geq F_2$. Hence, $O_{F_1}(\wv_1) \geq O_{F_2}(\wv_1)$. All together we have $O_{F_2}(\ti \wv_2) \geq O_{F_2}(\wv_1)$. Now the above equation becomes:
\begin{align*}
2 \norm{c}_\infty \sqrt{N} d_\H(\dom(F_2) , \dom(F_1)) +  \omega(d_\H(\dom(F_2) , \dom(F_1)))
&\geq O_{F_2}(\ti \wv_2) - O_{F_2}(\wv_2) \\
&\geq O_{F_2}(\wv_1) - O_{F_2}(\wv_2). 
\end{align*}
Hence by Lemma \ref{lem: strong convex} we get 
\begin{align*}
\norm{\wv_1- \wv_2}^2 \leq 8C_LN (2 \norm{c}_\infty \sqrt{N} d_\H(\dom(F_2) , \dom(F_1)) +  \omega(d_\H(\dom(F_2) , \dom(F_1))))
\end{align*}
as desired.
\end{proof}

\begin{rmk}
	
Note that the assumption that $F_2$ is uniformly continuous on its domain poses no real added assumption, see Proposition \ref{prop: uniform continuous}. 
	
\end{rmk}

\section{Regularizations}\label{sec: Regularizations}

In this section we prove our second main theorem, Theorem \ref{thm: regularize}, which shows that for any storage fee function, $F$ that splits there is a storage fee function that satisfies the assumptions of our main convergence theorem and yields an optimizer close to that of $F$. 

\begin{prop}\label{prop: uniform continuous}
	
Let $F_1$ be a storage fee function. Define the storage fee function $F_2$ by
\begin{align*}
F_2 = F_1 + \delta_{\{\wv: F_1(\wv) \leq 2 \norm{c}_\infty + \min_{\ti \wv} F_1(\ti \wv)\} }.
\end{align*}
Then $F_2$ is uniformly continuous on its domain and if $\wv_1, \wv_2$ are the minimizers of problems associated to $F_1, F_2$ respectively then $\wv_1 = \wv_2$. 

\end{prop}

\begin{proof}
	
 Note that since $\dom F_2 \subset \wvs$ and so $\dom F_2$ is bounded. Furthermore note that $F_2 \leq 2 \norm{c}_\infty + \min_{\ti \wv} F_1(\ti \wv)$ on its domain and so $F_2$ is bounded on its domain. Since $F_2$ is a closed, convex function that is bounded on its domain we see that $\dom F_2$ is closed. Hence $\dom F_2$ is compact. Since convex functions are continuous on their domain this shows that $F_2$ is uniformly continuous on its domain.

Next we need to show $\wv_1 = \wv_2$. First we will show that $F_1(\wv_1) \leq 2 \norm{c}_\infty + \min_{\ti \wv} F_1(\ti \wv)$. Note that for any $\ti \wv \in \wvs$
\begin{align*}
\abs{\mathcal{C}(\ti \weightvect) }
= \abs{\min_{S_\#\mu =\sum_{i=1}^N \tilde\weightvect^i \delta_{y_i}} \int c(x, S(x)) d\mu} 
\leq \min_{S_\#\mu =\sum_{i=1}^N \tilde\weightvect^i \delta_{y_i}} \int \norm{c}_\infty d\mu
= \norm{c}_\infty.
\end{align*}
Hence for any $\ti \wv \in \wvs$
\begin{align*}
F_1(\wv_1) 
= O_{F_1}(\wv_1) - \mathcal{C}(\wv_1) 
\leq O_{F_1}(\ti \wv) + \norm{c}_\infty
= F_1(\ti \wv) + \mathcal{C}(\ti \wv)  + \norm{c}_\infty
\leq 2 \norm{c}_\infty + F_1(\ti \wv),
\end{align*}
and so minimizing over $\ti \wv \in \wvs$ gives $F_1(\wv_1) \leq 2 \norm{c}_\infty + \min_{\ti \wv} F_1(\ti \wv)$. Hence $F_2(\wv_1) = F_1(\wv_1)$ and $O_{F_2}(\wv_1) = O_{F_1}(\wv_1)$. Since $F_2 \geq F_1$ pointwise we get for any $\ti \wv \in \wvs$, 
\begin{align*}
O_{F_2}(\ti \wv) \geq O_{F_1}(\ti \wv) \geq O_{F_1}(\wv_1) = O_{F_2}(\wv_1),
\end{align*}
and so $\wv_1$ is indeed the minimizer of the problem associated to $F_2$ and so $\wv_1 = \wv_2$. 
\end{proof}

\begin{cor}	\label{cor: uni cont split}
Suppose $F_1$ is a storage fee function that splits into functions $f_{i, 1}$. Define $f_{i,2}$ by
\begin{align*}
f_{i,2} = f_{i, 1} + \delta_{\{x: f_{i,1}(x) \leq 2 \norm{c}_\infty + \min_{\ti \wv} F_1(\ti \wv) - \sum_{j\neq i} \min_x f_{j,1}(x)\} }.
\end{align*}
Then the $f_{i,2}$ are proper, convex functions that are uniformly continuous on their domains. Furthermore if we define the storage fee function $F_2(\wv) = \sum_i f_{i,2}(\wv^i) + \delta_\wvs(\wv)$, then the minimizers of problems associated to $F_1, F_2$ are equal.
\end{cor}

\begin{proof}	
The proof is similar to that of Proposition \ref{prop: uniform continuous}. If $\wv_1$ is the minimizer in the problem associated to $F_1$ then all we need to show is that $f_{i,2}(\wv_1^i) < +\infty$ which is equivalent to showing 
\begin{align*}
f_{i,1}(\wv_1^i) \leq 2 \norm{c}_\infty + \min_{\ti \wv} F_1(\ti \wv) - \sum_{j\neq i} \min_x f_{j,1}(x).
\end{align*}
But this follows because
\begin{align*}
2 \norm{c}_\infty + \min_{\ti \wv} F_1(\ti \wv) \geq F_1(\wv_1) = \sum_j f_{j,1}(\wv_1^j) \geq f_{i,1}(\wv_1^i) + \sum_{j\neq i} \min_x f_{j,1}(x).
\end{align*}
\end{proof}

In the remainder of this section we discuss how to take an arbitrary storage fee function $F$ that splits and ``regularize'' it into a new storage fee function that satisfies the assumptions of Theorem \ref{thm: Newton Convergence}. Recall that we use $d_\H$ to denote Hausdorff distance. We start with a technical lemma.

\begin{lem} \label{lem: hausdorff hypercube}	
	Let $A = \prod_i [a_i,b_i]$ and $B = \prod_i [c_i, d_i]$ be hypercubes in $\R^N$. If $A \cap \wvs, B\cap \wvs \neq \emptyset$ then $d_\H(A \cap \wvs,B \cap \wvs) \leq 4 \sum_i \max(\abs{a_i - c_i}, \abs{b_i - d_i})$.
\end{lem}

\begin{proof}	
Define the hypercube $C = \prod_i[\min(a_i, c_i), \max(b_i, d_i)]$. Note that $A, B \subset C$. We will show that $d_\H(A \cap \wvs,C \cap \wvs) \leq 2 \sum_i \max(\abs{a_i - c_i}, \abs{b_i - d_i})$. Once this is done a symmetric argument will give the same bound for $d_\H(B \cap \wvs,C \cap \wvs)$. Our lemma will then follow from the triangle inequality (for Hausdorff distance). 
	
First we handle the case where $\sum_i a_i = 1$. Fix some $\wv \in C \cap \wvs$ and define $\ti \wv$ by $\ti \wv^i = a_i$ so that $\ti \wv \in A \cap \wvs$. Let $S = \{i: \wv^i > a_i \}$ Note that since $\sum_i \wv^i = 1$ we have 
\begin{align*}
\sum_{i \in S} \wv^i = 1 - \sum_{i \not\in S} \wv^i \leq 1 - \sum_{i \not\in S} \min(a_i, c_i) 
\leq 1 - \sum_{i \not\in S} \(a_i - \abs{a_i - c_i} \)
= \sum_{i \in S} a_i + \sum_{i \not\in S} \abs{a_i - c_i}
\end{align*}
and so
\begin{align*}
\norm{\wv-\ti \wv}_2
\leq \norm{\wv-\ti \wv}_1 
= \sum_i \abs{\wv^i - a_i} 
= \sum_{i\in S} (\wv^i - a_i) - \sum_{i\not \in S} (\wv^i - a_i) 
= 2 \sum_{i\in S} (\wv^i - a_i)
\leq 2 \sum_i \abs{a_i - c_i}.
\end{align*}
A similar result holds when $\sum_i b_i = 1$. Hence we may assume that $\sum_i a_i < 1 < \sum_i b_i$. Again choose some $\wv \in C \cap \wvs$ and define $\ti \wv$ by
\begin{align*}
\ti \wv^i =
\begin{cases}
\wv^i, & \text{ if } \wv^i \in [a_i, b_i] \\
a_i & \text{ if } \wv^i < a_i \\
b_i & \text{ if } \wv^i > b_i
\end{cases}
\end{align*}
Assume without loss of generality that $\sum_i \ti \wv^i \geq 1$. Then we define $\ha \wv$ by $\ha \wv^i = (1-t) \ti \wv^i + t a_i$ where $t = \frac{(\sum_i \ti \wv^i) - 1}{\sum_i (\ti \wv^i - a_i)}$. Note that $\ha \wv \in A \cap \wvs$. Furthermore we see that $\norm{\ti \wv - \wv}_1 \leq \sum_i \max (\abs{a_i - c_i}, \abs{b_i - d_i})$ and that
\begin{align*}
\norm{\ha \wv - \ti \wv}_1 
= \sum_i \ti \wv^i - \ha \wv^i
= (\sum_i \ti \wv^i) - 1 
= \sum_i (\ti \wv^i - \wv^i) \leq \norm{\ti \wv - \wv}_1 \leq \sum_i \max (\abs{a_i - c_i}, \abs{b_i - d_i})
\end{align*}
and so $\norm{\ha \wv - \wv}_2 \leq 2 \sum_i \max (\abs{a_i - c_i}, \abs{b_i - d_i})$. In either case we got $d_\H(A \cap \wvs,C \cap \wvs) \leq 2 \sum_i \max(\abs{a_i - c_i}, \abs{b_i - d_i})$ and so the proof follows.
\end{proof}

We are now ready to proceed to the proof of the second main theorem. 

\begin{proof}[Proof of Theorem \ref{thm: regularize}]
We modify $F$ one piece at a time in order to get all of the assumptions satisfied. First, set $F_1 = F$ and $f_{i,1} = f_i$. We define $f_{i,2}$ as in Corollary \ref{cor: uni cont split}. 

Our first task will be to deal with the possibility that some $\dom f_{i,2}$ might be a single point. The only way this is possibly is if $f_{i,2}$ is the indicator function of a point plus a constant, i.e. if $f_{i,2} = \delta_{\{y\}} + K$ for some $K \in \R$. In this case we set $f_{i,3} = \delta_{ [y-\eta, y+\eta] \cap [0,1]} + K$. Otherwise we define $f_{i,3} = f_{i,2}$. 

Now, $f_{i,3}$ is a convex function with bounded domain and by construction it is bounded on its domain. Hence its domain is a closed interval, say $\dom f_{i,3} = [a_i, b_i]$. Since $\dom f_{i,3}$ is not a single point we have $a_i < b_i$. Now in order to define $f_{i,4}$ we split into two cases. First we consider the case where $\sum_i a_i < 1 < \sum_i b_i$. Choose some $\eps> 0$ so that $\eps < \min(\frac{1- \sum_i a_i}{2N}, \eta)$ and $\eps < \min_i {b_i}$. Define $f_{i,4}$ by $f_{i,4} = f_{i,3} + \delta_{[\eps, 1]}$. Note that now if we let $c_i, d_i$ be so that $\dom f_{i,4} = [c_i, d_i]$, then $d_i = b_i$, and $c_i = \max(a_i, \eps)$. In particular $c_i \geq \eps > 0$ and 
\begin{align*}
\sum_i c_i \leq \sum_i (a_i + \eps) \leq \(\sum_i a_i\) + \frac{1-\sum_i a_i}{2} = \frac{1 + \sum_i a_i }{2} < 1.
\end{align*}
Next we deal with the case where $\sum_i a_i = 1$ or $\sum_i b_i = 1$. Without loss of generality assume that $\sum_i a_i = 1$. In this case we define $f_{i,4} = \delta_{[c_i, d_i]}$ where $c_i = \frac{a_i + \frac{\eta}{N}}{1+2\eta}$ and $d_i = a_i + \frac{\eta}{N}$. We note that $\min_i c_i \geq \frac{{\eta}}{N+2\eta N}$ and 
\begin{align*}
\sum_i c_i = \frac{1}{1+2\eta} \sum_i d_i < \sum_i d_i.
\end{align*}
Since the $f_{i,4}$ are uniformly continuous convex functions on a bounded interval, there are convex functions $f_{i,5}$ so that $\dom f_{i,5} = \dom f_{i,4}$, $\norm{f_{i,4} - f_{i,5}}_{L^\infty(\dom f_{i,4})} < \eta$ and $f_{i,5}$ are smooth on the interior of their domains. One can construct these $f_{i,5}$ by first taking a polygonal approximate of the $f_{i,4}$ and then smoothing it (see \cite[Theorem 2]{Koliha03}).  

Recall that $\dom f_{i,5} = \dom f_{i,4} = [c_i, d_i]$. We now define $f_{i,6}(x) = f_{i,5}(x) - \eta \sqrt{(d_i-x)(x-c_i)}$. The AM-GM inequality shows that $\norm{f_{i,5} - f_{i,6}}_{L^\infty([c_i, d_i])} \leq \frac{\eta}{2} (d_i - c_i)$ and we see that $f_{i,6}$ is strongly convex with parameter $\eta\frac{2}{d_i - c_i}$. Furthermore, it is clear that $f_{i,6}$ is essentially smooth, and smooth on the interior of its domain. In particular $f_{i,6}'''$ is locally bounded on the interior of $\dom f_{i,6}$. Furthermore $\dom f_i = [c_i, d_i]$ and we saw above that $\min c_i > 0$ and $\sum_i c_i < 1 < \sum d_i$. 

Now we set $F_j = \sum_i f_{i,j}(\wv^i) + \delta_\wvs(\wv)$ for $j \in \{2,3,4,5, 6\}$. $F_6$ is the promised $\ti F$, i.e. we see from the above paragraph that $F_6$ satisfies the assumptions of Theorem \ref{thm: Newton Convergence}. All that is left is to prove that the minimizer of the problem associated to $F_6$ is close to that to $F_1$. 

Let $\wv_i$ be the minimizers associated to the $F_i$. We have $\wv_1 = \wv_2$ by Corollary \ref{cor: uni cont split}. Note that by Lemma \ref{lem: hausdorff hypercube} we have that $d_\H(\dom F_2, \dom F_3) \leq 4N \eta$ and so by Proposition \ref{prop: hausdorff stab} we get that $\norm{\wv_2 - \wv_3}^2 \leq 64C_LN^{5/2}\norm{c}_{\infty} \eta +  8 C_L N\omega(4N\eta)$ where $\omega$ is the modulus of continuity of $F_2$. 

For controlling $\norm{\wv_3-\wv_4}$, recall that we split into two separate cases. First we look at the case where we had $\sum_i a_i < 1 < \sum_i b_i$. In this case by Lemma \ref{lem: hausdorff hypercube} we have that $d_\H(\dom F_3, \dom F_4) \leq 4N \eps \leq 4N \eta$ and so we get again that $\norm{\wv_3 - \wv_4}^2 \leq 64C_LN^{5/2}\norm{c}_{\infty} \eta +  8 C_L N\omega(4N\eta)$. Note that we have used that $\omega$ is also the modulus of continuity of $F_3$.
 Now if we are in the other case, i.e. $\sum_i a_i = 1$ then note that $\dom F_3 = \{a\}$ where $a = (a_1, \dots, a_N) \in \R^N$. Hence we can apply Lemma \ref{lem: hausdorff hypercube} with $A = \{a\}$ and $B = \prod_i [c_i, d_i]$ to get $d_\H(\dom F_3, \dom F_4) \leq 4 \eta$ and so $\norm{\wv_3 - \wv_4}^2 \leq 64C_LN^{3/2}\norm{c}_{\infty} \eta +  8 C_L N\omega(4\eta)$.

Finally by Proposition \ref{prop: uniform stab} that $\norm{\wv_4-\wv_6} \leq 4 \sqrt{C_L N \sum_{i}(\eta + \eta\frac{d_i - c_i}{2})}$. Since $\omega$ depends only on the initial $F$ (and not on $\eta$) and $\abs{d_i - c_i} \leq 1$ (and so is also independent of $\eta$) we get the desired result. 	
\end{proof}

\section{Acknowledgments}
I would like to thank Kitagawa for helpful comments and suggestions on a previous version of this manuscript.

\begin{appendix}
	
\section {Bounds on $\psi's$}\label{sec: Appendix Bounds on psi's}

In this section we give prove a lemma bounding the difference between different coordinates of a $\psi$ that generates a Laguerre diagram where each cell has positive mass. 

\begin{lem}\label{lem: psi bound}
	If $G(\psi)^j > 0$ then for all $k$ we have $\psi^j - \psi^k \leq  2 \norm{c}_{\infty}$. In particular if $\psi \in \mathcal{K}^0$ the for all $j,k$ we have $\abs{\psi^j - \psi^k} \leq 2 \norm{c}_{\infty}$.
\end{lem}

\begin{proof}	
	Suppose for sake of contradiction that $\psi^j - \psi^k >  2 \norm{c}_{\infty}$. Then	
	\begin{align*}
	\Lag_j(\psi) 
	&= \{x\in X\mid c(x, y_j)+\psi^j= \min_{i} c(x,y_i) + \psi^i \} \\
	&\subset \{x\in X\mid c(x, y_j)+\psi^j \leq c(x,y_{k}) + \psi^{k} \} \\
	&= \{x\in X\mid \psi^j - \psi^k \leq c(x,y_{k}) - c(x, y_j) \} \\
	&\subset \{x\in X\mid \psi^j - \psi^k \leq 2 \norm{c}_{\infty} \} 
	= \emptyset
	\end{align*}	
	and so $G(\psi)^j = 0$ which is a contradiction.
\end{proof}

\end{appendix}

\bibliographystyle{alpha}
\bibliography{snowshovelingalg}

\end{document}